\documentclass[preprint,12pt]{elsarticle}
\usepackage{amsthm,amsmath,amssymb}
\usepackage[colorlinks=true,citecolor=black,linkcolor=black,urlcolor=blue]{hyperref}
\usepackage{graphicx}
\usepackage{float}
\usepackage{mathtools}
\usepackage{epstopdf}
\usepackage{epsfig}
\usepackage{subfigure}
\usepackage{enumitem}
\usepackage{tikz}
\usepackage{nicematrix}
\usetikzlibrary{calc,decorations.pathmorphing,arrows.meta,fit,shapes.geometric}
\usepackage{arydshln}
\usepackage{blkarray}

\newcommand{\arxiv}[1]{\href{http://arxiv.org/abs/#1}{\texttt{arXiv:#1}}}

\theoremstyle{plain}
\newtheorem{theorem}{Theorem}
\newtheorem{lemma}[theorem]{Lemma}

\newtheorem{proposition}[theorem]{Proposition}

\theoremstyle{definition}
\newtheorem{definition}[theorem]{Definition}

\newtheorem{example}[theorem]{Example}

\newtheoremstyle{myremark}
{3pt}   
{3pt}   
{\normalfont} 
{}      
{\bfseries} 
{.}     
{.5em}  
{}

\theoremstyle{myremark}
\newtheorem{remark}{Remark}

\newcommand{\GF}{\mathrm{GF}}

\newcommand{\ptr}{P^{\times}}

\title{Recurrence and coefficient inequality for the partial Petrial polynomial of graphs}

\author{Ruiqing Feng$^{1}$, Xia Guo$^{2}$,
Qi Yan$^{1}$\footnote{Corresponding author.}\\
\small $^{1}$ School of Mathematics and Statistics, Lanzhou University, PR China\\
\small $^{2}$ Xi'an Research Institute of High-Tech Hongqing Town, PR China\\
\small{\tt Email: fengrq2024@lzu.edu.cn; xguomath@163.com; yanq@lzu.edu.cn}\\
}

\date{}
\journal{}

\begin{document}
\begin{frontmatter}

\begin{abstract}
The partial Petrial polynomial of a ribbon graph, introduced by Gross, Mansour and Tucker, enumerates partial Petrials by Euler genus. Recently, Deng, Jin and Yan defined an analogue for grafts and showed that it can be expressed as a rank-generating function of an adjacency matrix. In this paper we first prove a recurrence relation that reduces the partial Petrial polynomial of a graph with respect to an arbitrary edge, expressing it as a sum of three terms involving graphs obtained by local complementation and edge pivoting. This recurrence extends the known leaf-reduction formula to vertices of any positive degree. Second, using this recurrence we compare the lowest and highest degree coefficients of the polynomial. We prove that the lowest coefficient is always at most the highest coefficient, and that equality holds if and only if the graph has no edges.
\end{abstract}


\end{frontmatter}

\section{Introduction}

The Petrial operation, introduced by Wilson \cite{Wilson1979} in the study of regular maps, is a duality operation that replaces each edge with a new edge that is orthogonal to the original edge with respect to the embedding. When applied only to a given subset of edges, one obtains a partial Petrial \cite{EllisMoffatt2012}. Gross, Mansour and Tucker \cite{GrossMansourTucker2021} introduced partial-twuality polynomials for ribbon graphs, including the partial Petrial polynomial, as generating functions that enumerate the Euler genus of all partial Petrials. These polynomials have since been investigated from various perspectives, including algebraic, combinatorial and matrix approaches.

Recently, Yan and Jin \cite{YanJin2024} studied partial-twuality polynomials of delta-matroids via binary matrices and their associated intersection graphs. Yan and Li \cite{YanLi2025} obtained explicit formulas for complete graphs and paths. Deng, Jin and Yan \cite{DengJinYan2026} developed a unified matrix framework for these polynomials, establishing product formulas, degree bounds and reduction formulas. In particular, they showed that the partial Petrial polynomial of a graft, a pair consisting of a simple graph and a vertex subset, can be studied through the rank of a certain adjacency matrix. Yu, Hao, Liu and Li \cite{YuHaoLiuLi2026} gave a rank-based expression for the partial Petrial polynomial of ribbon graphs and established a four-term relation.

Local complementation and edge pivots, which originated in the work of Kotzig \cite{Kotzig1968} and Bouchet \cite{Bouchet1988}, play a key role in the proofs of our main results. Edge pivots were independently rediscovered by Arratia, Bollobás and Sorkin in the context of the interlace polynomial \cite{ArratiaBollobasSorkin2000,ArratiaBollobasSorkin2004}. Moffatt \cite{Moffatt2019,Moffatt2023} gave a comprehensive exposition of their connections to delta-matroids and matrix pivots. Moreover, Feng, Yan and Zheng \cite{FengYanZheng2026} recently used grafts and local complementation to characterize circle graphs with binomial partial Petrial polynomials.

Our first main result is a recurrence for the partial Petrial polynomial of a simple graph with respect to an arbitrary edge. This recurrence expresses the polynomial of a graph in terms of polynomials of three smaller graphs obtained via local complementation and edge pivoting. It extends a known leaf-reduction formula \cite{DengJinYan2026} to vertices of arbitrary degree.

Using this recurrence, we then compare the two extremal coefficients of the polynomial. Denote by $\ell_G$ and $h_G$ the coefficients of the lowest and highest degree terms, respectively. We prove that $\ell_G\le h_G$ for every simple graph $G$, and that equality holds if and only if $G$ has no edges. This gives a  characterisation of when the polynomial is symmetric (i.e., its coefficients form a palindromic sequence).

The paper is organised as follows. Section~2 recalls the necessary definitions, including ribbon graphs, grafts, local complementation and edge pivots, as well as the principal pivot transform. Section~3 states and proves the recurrence, discusses the leaf case and verifies the formula for complete graphs. Section~4 establishes the inequality between the extremal coefficients and classifies the equality case.

\section{Preliminaries}

Throughout this paper, all graphs are simple, i.e., without loops or multiple edges. All matrices and matrix operations are taken over the binary field $\GF(2)$. 

\subsection{Partial Petrial polynomials of ribbon graphs and grafts}

We begin by briefly introducing ribbon graphs, which provide an equivalent description of cellularly embedded graphs.

\begin{definition}[\cite{BollobasRiordan2002}]
A {\it ribbon graph} $G=(V(G), E(G))$ is a $($orientable or non-orientable$)$ surface with boundary,
represented as the union of two sets of topological discs: a set $V(G)$ of vertices and a set $E(G)$ of edges with the following properties.
\begin{description}
\item[\rm (1)] The vertices and edges intersect in disjoint line segments.
\item[\rm (2)] Each such line segment lies on the boundary of precisely one vertex and precisely one edge.
\item[\rm (3)] Every edge contains exactly two such line segments.
\end{description}
\end{definition}

\begin{definition}[\cite{EllisMoffatt2012}]
Let \(G\) be a ribbon graph and
\(A\subseteq E(G)\). The \emph{partial Petrial} of \(G\) with respect to
\(A\), denoted \(G^{\times|A}\), is the ribbon graph obtained from \(G\) by
adding a half-twist to each edge in \(A\).
\end{definition}

The partial Petrial polynomial was originally defined for ribbon graphs as follows.

\begin{definition}[\cite{GrossMansourTucker2021}]
\label{def:ribbon-partial-petrial-polynomial}
The \emph{partial Petrial polynomial}
\({}^{\partial}\varepsilon^{\times}_{G}(z)\) of a ribbon graph \(G\) is the
generating function
\[
{}^{\partial}\varepsilon^{\times}_{G}(z)
:=
\sum_{A\subseteq E(G)}
z^{\varepsilon(G^{\times|A})},
\]
where \(\varepsilon(G^{\times|A})\) denotes the Euler genus of
\(G^{\times|A}\).
\end{definition}

A \emph{graft} \((G,L_G)\) consists of a simple graph \(G\) and a subset
\(L_G\subseteq V(G)\). The adjacency matrix \(\mathbf{A}(G,L_G)\) of a graft
\((G,L_G)\) is the matrix over \(\GF(2)\) whose rows and columns are indexed
by the vertices of \(G\). For distinct vertices \(u,v\), the \((u,v)\)-entry
is \(1\) if \(u\) and \(v\) are adjacent in \(G\), and \(0\) otherwise.
The \((v,v)\)-entry is \(1\) if \(v\in L_G\), and \(0\) otherwise.
When \(L_G=\emptyset\), we simply write \(\mathbf{A}_G:=\mathbf{A}(G,\emptyset)\)
and call it the adjacency matrix of \(G\). 

Recently, Deng, Jin and Yan \cite{DengJinYan2026} introduced the partial Petrial polynomial for grafts.

\begin{definition}[\cite{DengJinYan2026}]
\label{def:graft-partial-petrial-polynomial}
Let \((G,L_G)\) be a graft. The \emph{partial Petrial polynomial} of
\((G,L_G)\) is defined by
\[
P^\times_{(G,L_G)}(z)
:=
\sum_{F\subseteq V(G)}
z^{\operatorname{rank}\mathbf{A}(G,L_G\Delta F)},
\]
where \(L_G\Delta F\) denotes the symmetric difference of \(L_G\) and
\(F\).
\end{definition}

Let $e$ be an edge of a ribbon graph $G$. If $e$ is a loop at the vertex disc $v$ and $e\cup v$ is homeomorphic
to a M\"obius band, then we call $e$ a {\it non-orientable loop}. Otherwise it is said to be an {\it orientable loop}. A \emph{bouquet} is a ribbon graph having exactly one vertex. Two loops in a bouquet are said to be \emph{interlaced} if, when travelling along the boundary of the unique vertex, their ends are encountered in an alternating order. The \emph{intersection graph} \(I(B)\) of a bouquet \(B\) is the simple graph whose vertex set is \(E(B)\), and in which two vertices (corresponding to loops of \(B\)) are adjacent if and only if the corresponding loops are interlaced in \(B\). Given a bouquet \(B\), let \(L_{I(B)} \subseteq E(B)\) be the set of its non-orientable loops. Then the graft \(\bigl(I(B),\, L_{I(B)}\bigr)\) is the \emph{intersection graft} of \(B\). 

\begin{proposition}[\cite{DengJinYan2026}]
\label{thm:bouquet-intersection-graft-polynomial}
Let \(B\) be a bouquet, and let \((I(B),L_{I(B)})\) be its intersection graft.
Then
\[
{}^{\partial}\varepsilon^{\times}_{B}(z)
=
P^\times_{(I(B),L_{I(B)})}(z).
\]
\end{proposition}

Proposition~\ref{thm:bouquet-intersection-graft-polynomial} shows that the
partial Petrial polynomial of a bouquet can be expressed entirely in terms of
its intersection graft. Consequently, \(P^\times_{(G,L_G)}(z)\) can be viewed as a matrix representation of the partial Petrial polynomial of a bouquet.

The partial Petrial polynomial of a graft is independent of the initial subset \(L_G\) \cite{DengJinYan2026}. Hence \[P^\times_{(G,L_G)}(z)=P^\times_{(G,\emptyset)}(z).\]
For a simple graph \(G\), we define its partial Petrial polynomial as
\[
P_G^\times(z) := P^\times_{(G,\emptyset)}(z) = \sum_{F\subseteq V(G)} z^{\operatorname{rank}\mathbf{A}(G,F)}.
\]

\subsection{Local complementation and edge pivoting}

We next recall local complementation. For a vertex $v\in V(G)$, let $N_G(v)$ denote the set of neighbours of $v$ in $G$. Note that $v\notin N_G(v)$.

\begin{definition}[\cite{Kotzig1968}]\label{def:local-complementation}
Let $G$ be a simple graph and $v\in V(G)$. The \emph{local complementation} at $v$, denoted by $G*v$, is the graph obtained from $G$ by replacing the induced subgraph on the neighbourhood $N_G(v)$ with its complement. Equivalently, $G*v$ is formed by toggling all adjacencies between vertices in $N_G(v)$, that is, replacing edges with non-edges and vice versa within $N_G(v)$. We further define
\[
G\underline{*}v:=(G*v)\setminus\{v\}.
\]
\end{definition}

We also define local complementation for grafts as follows. 

\begin{definition}[\cite{Moffatt2019}]\label{def:graft-local-complementation}
Let $G$ be a simple graph and $L_G\subseteq V(G)$. For any vertex $v\in L_G$, the \emph{local complementation} at $v$ is the operation on the graft $(G,L_G)$ defined by
\[
(G,L_G)\mapsto (G*v,\ L_G\Delta N_G(v)).
\]
The \emph{local complementation deletion} at $v$ is the operation
\[
(G,L_G)\underline{*}v
:=
\bigl(G\underline{*}v,\ (L_G\setminus\{v\})\Delta N_G(v)\bigr).
\]
A graft \((H,L_H)\) is a \emph{local complementation minor} of \((G,L_G)\) if it
can be obtained from \((G,L_G)\) by a sequence of local complementation
deletion operations.
\end{definition}

The edge pivot operation was introduced by Bouchet in the study of isotropic systems \cite{Bouchet1988} and multimatroids \cite{Bouchet2001}. Independently, it was rediscovered by Arratia, Bollob\'as and Sorkin in their work on the interlace polynomial \cite{ArratiaBollobasSorkin2000,ArratiaBollobasSorkin2004}.

\begin{definition}\label{def:edge-pivot}
Let \(G\) be a simple graph and \(uv\in E(G)\). Partition the vertices other
than \(u\) and \(v\) into four classes:
\begin{itemize}
\item[(1)] vertices adjacent to \(u\) but not \(v\);
\item[(2)] vertices adjacent to \(v\) but not \(u\);
\item[(3)] vertices adjacent to both \(u\) and \(v\);
\item[(4)] vertices adjacent to neither \(u\) nor \(v\).
\end{itemize}
The \emph{pivot} of the edge \(uv\) is the graph \(G\wedge uv\), constructed from \(G\)
as follows. For any vertex pair \(x,y\) where \(x\) is in one of the classes
(1)--(3), and \(y\) is in a different class (1)--(3), toggle the pair \(xy\)
in the edge set, so if \(xy\) was an edge, make it a non-edge; and if \(xy\)
was a non-edge, make it an edge. Finally, switch the names of the vertices
\(u\) and \(v\). See Figure~\ref{fig:edge-pivot} \cite{Moffatt2023}.
\end{definition}

\begin{figure}[H]
\centering
\begin{tikzpicture}[
    x=0.01cm,
    y=-0.01cm,
    scale=1.25,
    transform shape,
    thinlines/.style={line width=0.55pt},
    ellipsebox/.style={draw, line width=0.65pt, fill=white},
    settext/.style={font=\scriptsize},
    toplabel/.style={font=\scriptsize},
    captext/.style={font=\footnotesize},
    toggletext/.style={font=\scriptsize},
    leftstep/.style={
        line width=2.0pt,
        line cap=round,
        line join=round
    },
    rightstep/.style={
        line width=0.95pt,
        line cap=round,
        line join=round
    },
    leftzig/.style={
        line width=1.7pt,
        decorate,
        decoration={zigzag, segment length=5.5pt, amplitude=1.5pt},
        line cap=round
    },
    rightzig/.style={
        line width=0.95pt,
        decorate,
        decoration={zigzag, segment length=5.5pt, amplitude=1.35pt},
        line cap=round
    }
]


\coordinate (uA) at (255,87);
\coordinate (vA) at (407,87);

\draw[thinlines] (uA) -- (vA);

\foreach \x/\y in {
    176/277,
    207/249,
    238/277
}{
    \draw[thinlines] (uA) -- (\x,\y);
}

\foreach \x/\y in {
    290/169,
    309/164,
    322/160
}{
    \draw[thinlines] (uA) -- (\x,\y);
}

\foreach \x/\y in {
    340/160,
    353/164,
    372/169
}{
    \draw[thinlines] (vA) -- (\x,\y);
}

\foreach \x/\y in {
    422/277,
    453/249,
    484/277
}{
    \draw[thinlines] (vA) -- (\x,\y);
}

\draw[ellipsebox] (207,285) ellipse [x radius=56, y radius=35];
\draw[ellipsebox] (331,194) ellipse [x radius=57, y radius=34];
\draw[ellipsebox] (453,285) ellipse [x radius=56, y radius=35];
\draw[ellipsebox] (331,360) ellipse [x radius=56, y radius=35];

\draw[leftstep]
    (255,260) -- (255,249) -- (268,249) -- (268,237)
    -- (281,237) -- (281,225) -- (294,225);

\draw[leftstep]
    (407,260) -- (407,249) -- (394,249) -- (394,237)
    -- (381,237) -- (381,225) -- (368,225);

\draw[leftzig] (274,286) -- (388,286);

\node[settext] at (207,286) {$S_u$};
\node[settext] at (331,195) {$S_{uv}$};
\node[settext] at (453,286) {$S_v$};

\node[toggletext] at (331,255) {toggle};

\fill (uA) circle [radius=10];
\fill (vA) circle [radius=10];

\node[toplabel] at (255,65) {$u$};
\node[toplabel] at (407,65) {$v$};

\node[captext] at (331,427) {{\bfseries (a)} $G$};


\coordinate (vB) at (752,87);
\coordinate (uB) at (904,87);

\draw[thinlines] (vB) -- (uB);

\foreach \x/\y in {
    673/277,
    704/249,
    735/277
}{
    \draw[thinlines] (vB) -- (\x,\y);
}

\foreach \x/\y in {
    787/169,
    806/164,
    819/160
}{
    \draw[thinlines] (vB) -- (\x,\y);
}

\foreach \x/\y in {
    837/160,
    850/164,
    869/169
}{
    \draw[thinlines] (uB) -- (\x,\y);
}

\foreach \x/\y in {
    920/277,
    951/249,
    982/277
}{
    \draw[thinlines] (uB) -- (\x,\y);
}

\draw[ellipsebox] (704,285) ellipse [x radius=56, y radius=35];
\draw[ellipsebox] (828,194) ellipse [x radius=57, y radius=34];
\draw[ellipsebox] (951,285) ellipse [x radius=56, y radius=35];
\draw[ellipsebox] (828,360) ellipse [x radius=56, y radius=35];

\draw[rightstep]
    (760,263) -- ++(0,-9) -- ++(11,0)
    -- ++(0,-9) -- ++(11,0)
    -- ++(0,-9) -- ++(11,0);

\draw[rightstep]
    (768,271) -- ++(0,-9) -- ++(11,0)
    -- ++(0,-9) -- ++(11,0)
    -- ++(0,-9) -- ++(11,0);

\draw[rightstep]
    (896,263) -- ++(0,-9) -- ++(-11,0)
    -- ++(0,-9) -- ++(-11,0)
    -- ++(0,-9) -- ++(-11,0);

\draw[rightstep]
    (888,271) -- ++(0,-9) -- ++(-11,0)
    -- ++(0,-9) -- ++(-11,0)
    -- ++(0,-9) -- ++(-11,0);

\draw[rightzig] (772,288) -- (884,288);
\draw[rightzig] (772,300) -- (884,300);

\node[settext] at (704,286) {$S_u$};
\node[settext] at (828,195) {$S_{uv}$};
\node[settext] at (951,286) {$S_v$};

\fill (vB) circle [radius=10];
\fill (uB) circle [radius=10];

\node[toplabel] at (752,65) {$v$};
\node[toplabel] at (904,65) {$u$};

\node[captext] at (828,427) {{\bfseries (b)} $G\wedge uv$};

\end{tikzpicture}
\caption{The edge pivot operation: edges between the three sets \(S_u\), \(S_v\), and \(S_{uv}\) are toggled, and the names of \(u\) and \(v\) are interchanged.}
\label{fig:edge-pivot}
\end{figure}
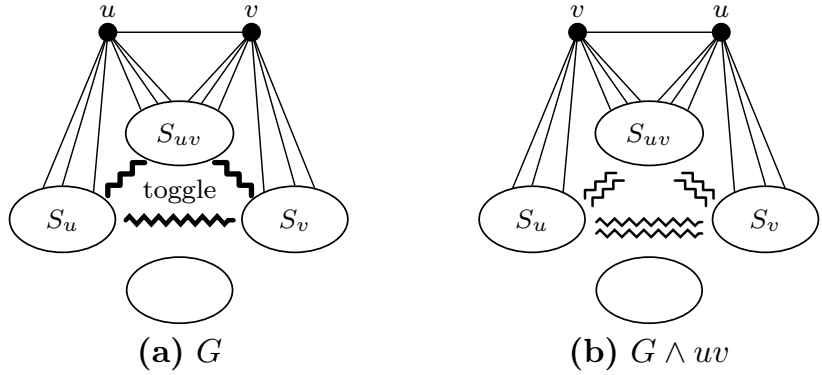

The matrix form of the edge pivot is the principal pivot transform, defined as follows. Let \(M\) be a square matrix whose rows and columns are indexed by a finite
set \(V\), and let \(X\subseteq V\). Suppose that the principal submatrix
\(M[X]\) is nonsingular. With respect to the decomposition
\(V=X\cup (V\setminus X)\), write
\[
M=
\begin{pmatrix}
	A & B\\
	C & D
\end{pmatrix},
\qquad A=M[X].
\]
The \emph{principal pivot transform} of \(M\) with respect to \(X\), denoted by
\(M*X\), is defined by
\[
M*X=
\begin{pmatrix}
	A^{-1} & A^{-1}B\\
	-CA^{-1} & D-CA^{-1}B
\end{pmatrix}.
\]

\section{A recurrence for the partial Petrial polynomial}

To obtain the recurrence, we first present several necessary lemmas and introduce two auxiliary graphs.

\begin{proposition}[\cite{FengYanZheng2026}]\label{prop:corank-invariance}
Let \((G,L_G)\) and \((H,L_H)\) be grafts. If \((H,L_H)\) is a local
complementation minor of \((G,L_G)\), then
\[
\operatorname{corank} \mathbf{A}(G,L_G)
=
\operatorname{corank} \mathbf{A}(H,L_H).
\]
\end{proposition}

\begin{lemma}
	\label{lem:local-complement-rank}
	Let \(G\) be a simple graph and \(v\in V(G)\). If \(F\subseteq V(G)\) contains \(v\), then
\[
\operatorname{rank} \mathbf{A}(G,F) = \operatorname{rank} \mathbf{A}\bigl((G,F)\underline{*}v\bigr)+1.
\]
\end{lemma}

\begin{proof}
Since \(v\in F\), Proposition~\ref{prop:corank-invariance} gives
    \[
    \operatorname{corank} \mathbf{A}(G,F) = \operatorname{corank} \mathbf{A}\bigl((G,F)\underline{*}v\bigr).
    \]
    Hence
    \[
    |V(G)| - \operatorname{rank} \mathbf{A}(G,F) = (|V(G)|-1) - \operatorname{rank} \mathbf{A}\bigl((G,F)\underline{*}v\bigr),
    \]
    which simplifies to
    \[
    \operatorname{rank} \mathbf{A}(G,F) = \operatorname{rank} \mathbf{A}\bigl((G,F)\underline{*}v\bigr) + 1.
    \]

\end{proof}

Given a graph \(G\) and an edge \(uv\in E(G)\), we define two auxiliary graphs:
\[
G_{u,v}^{0} := (G\wedge uv)\setminus\{u,v\},\qquad
G_{u,v}^{1} := ((G*u)\wedge uv)\setminus\{u,v\}.
\]

\begin{example}
Figure~\ref{fig:example-G0-G1} illustrates the constructions of \(G_{u,v}^{0}\) and \(G_{u,v}^{1}\).
\end{example}

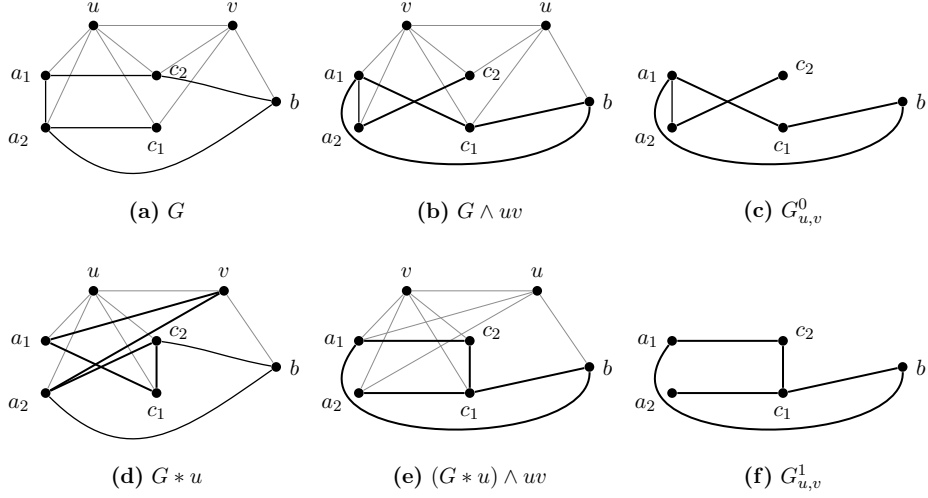
\begin{figure}[H]
\centering
\resizebox{0.90\textwidth}{!}{%
\begin{tikzpicture}[
    scale=0.76,
    every node/.style={font=\small},
    vtx/.style={circle, fill=black, inner sep=1.7pt},
    baseedge/.style={line width=0.45pt, black!45},
    edge/.style={line width=0.65pt},
    newedge/.style={line width=1.0pt},
    paneltitle/.style={font=\small}
]


\begin{scope}[xshift=0cm,yshift=0cm]

\node[vtx,label=above:\(u\)] (uA) at (0,2.65) {};
\node[vtx,label=above:\(v\)] (vA) at (3.20,2.65) {};

\node[vtx,label=left:\(a_1\)] (a1A) at (-1.10,1.50) {};
\node[vtx] (a2A) at (-1.10,0.30) {};
\node[anchor=east] at (-1.28,0.03) {\(a_2\)};

\node[vtx] (c2A) at (1.45,1.50) {};
\node[anchor=south west] at (1.55,1.25) {\(c_2\)};
\node[vtx,label=below:\(c_1\)] (c1A) at (1.45,0.30) {};

\node[vtx,label=right:\(b\)] (bA) at (4.20,0.90) {};

\draw[baseedge] (uA)--(vA);
\draw[baseedge] (uA)--(a1A);
\draw[baseedge] (uA)--(a2A);
\draw[baseedge] (uA)--(c1A);
\draw[baseedge] (uA)--(c2A);

\draw[baseedge] (vA)--(bA);
\draw[baseedge] (vA)--(c1A);
\draw[baseedge] (vA)--(c2A);

\draw[edge] (a1A)--(a2A);
\draw[edge] (a1A)--(c2A);
\draw[edge] (a2A)--(c1A);
\draw[edge] (a2A) to[out=-45,in=215,looseness=1.25] (bA);
\draw[edge] (c2A) to[out=-8,in=170] (bA);

\node[paneltitle] at (1.45,-1.65) {{\bfseries (a)} \(G\)};

\end{scope}

\begin{scope}[xshift=7.20cm,yshift=0cm]

\node[vtx,label=above:\(v\)] (uB) at (0,2.65) {};
\node[vtx,label=above:\(u\)] (vB) at (3.20,2.65) {};

\node[vtx,label=left:\(a_1\)] (a1B) at (-1.10,1.50) {};
\node[vtx] (a2B) at (-1.10,0.30) {};
\node[anchor=east] at (-1.28,0.03) {\(a_2\)};

\node[vtx] (c2B) at (1.45,1.50) {};
\node[anchor=south west] at (1.55,1.20) {\(c_2\)};
\node[vtx,label=below:\(c_1\)] (c1B) at (1.45,0.30) {};

\node[vtx,label=right:\(b\)] (bB) at (4.20,0.90) {};

\draw[baseedge] (uB)--(vB);
\draw[baseedge] (uB)--(a1B);
\draw[baseedge] (uB)--(a2B);
\draw[baseedge] (uB)--(c1B);
\draw[baseedge] (uB)--(c2B);

\draw[baseedge] (vB)--(bB);
\draw[baseedge] (vB)--(c1B);
\draw[baseedge] (vB)--(c2B);

\draw[edge] (a1B)--(a2B);
\draw[newedge] (a1B)--(c1B);
\draw[newedge] (a2B)--(c2B);
\draw[newedge] (c1B)--(bB);
\draw[newedge]
    (a1B) .. controls (-3.10,-1.10) and (4.25,-1.05) .. (bB);

\node[paneltitle] at (1.45,-1.65) {{\bfseries (b)} \(G\wedge uv\)};

\end{scope}

\begin{scope}[xshift=14.40cm,yshift=0cm]

\node[vtx,label=left:\(a_1\)] (a1C) at (-1.10,1.50) {};
\node[vtx] (a2C) at (-1.10,0.30) {};
\node[anchor=east] at (-1.28,0.03) {\(a_2\)};

\node[vtx] (c2C) at (1.45,1.50) {};
\node[anchor=south west] at (1.55,1.35) {\(c_2\)};
\node[vtx,label=below:\(c_1\)] (c1C) at (1.45,0.30) {};

\node[vtx,label=right:\(b\)] (bC) at (4.20,0.90) {};

\draw[edge] (a1C)--(a2C);
\draw[newedge] (a1C)--(c1C);
\draw[newedge] (a2C)--(c2C);
\draw[newedge] (c1C)--(bC);
\draw[newedge]
    (a1C) .. controls (-3.10,-1.10) and (4.25,-1.05) .. (bC);

\node[paneltitle] at (1.45,-1.65) {{\bfseries (c)} \(G_{u,v}^{0}\)};

\end{scope}


\begin{scope}[xshift=0cm,yshift=-6.10cm]

\node[vtx,label=above:\(u\)] (uD) at (0,2.65) {};
\node[vtx,label=above:\(v\)] (vD) at (3.00,2.65) {};

\node[vtx,label=left:\(a_1\)] (a1D) at (-1.10,1.50) {};
\node[vtx] (a2D) at (-1.10,0.30) {};
\node[anchor=east] at (-1.28,0.03) {\(a_2\)};

\node[vtx] (c2D) at (1.45,1.50) {};
\node[anchor=south west] at (1.55,1.35) {\(c_2\)};
\node[vtx,label=below:\(c_1\)] (c1D) at (1.45,0.30) {};

\node[vtx,label=right:\(b\)] (bD) at (4.20,0.90) {};

\draw[baseedge] (uD)--(vD);
\draw[baseedge] (uD)--(a1D);
\draw[baseedge] (uD)--(a2D);
\draw[baseedge] (uD)--(c1D);
\draw[baseedge] (uD)--(c2D);

\draw[baseedge] (vD)--(bD);
\draw[newedge] (vD)--(a1D);
\draw[newedge] (vD)--(a2D);

\draw[newedge] (a1D)--(c1D);
\draw[newedge] (a2D)--(c2D);
\draw[newedge] (c1D)--(c2D);
\draw[edge] (a2D) to[out=-45,in=215,looseness=1.25] (bD);
\draw[edge] (c2D) to[out=-8,in=170] (bD);

\node[paneltitle] at (1.45,-1.65) {{\bfseries (d)} \(G*u\)};

\end{scope}

\begin{scope}[xshift=7.20cm,yshift=-6.10cm]

\node[vtx,label=above:\(v\)] (uE) at (0,2.65) {};
\node[vtx,label=above:\(u\)] (vE) at (3.00,2.65) {};

\node[vtx,label=left:\(a_1\)] (a1E) at (-1.10,1.50) {};
\node[vtx] (a2E) at (-1.10,0.30) {};
\node[anchor=east] at (-1.28,0.03) {\(a_2\)};

\node[vtx] (c2E) at (1.45,1.50) {};
\node[anchor=south west] at (1.55,1.35) {\(c_2\)};
\node[vtx,label=below:\(c_1\)] (c1E) at (1.45,0.30) {};

\node[vtx,label=right:\(b\)] (bE) at (4.20,0.90) {};

\draw[baseedge] (uE)--(vE);
\draw[baseedge] (uE)--(a1E);
\draw[baseedge] (uE)--(a2E);
\draw[baseedge] (uE)--(c1E);
\draw[baseedge] (uE)--(c2E);

\draw[baseedge] (vE)--(a1E);
\draw[baseedge] (vE)--(a2E);
\draw[baseedge] (vE)--(bE);

\draw[newedge] (c1E)--(c2E);
\draw[newedge] (a1E)--(c2E);
\draw[newedge] (a2E)--(c1E);
\draw[newedge] (c1E)--(bE);
\draw[newedge]
    (a1E) .. controls (-3.10,-1.10) and (4.25,-1.05) .. (bE);

\node[paneltitle] at (1.45,-1.65) {{\bfseries (e)} \((G*u)\wedge uv\)};

\end{scope}

\begin{scope}[xshift=14.40cm,yshift=-6.10cm]

\node[vtx,label=left:\(a_1\)] (a1F) at (-1.10,1.50) {};
\node[vtx] (a2F) at (-1.10,0.30) {};
\node[anchor=east] at (-1.28,0.03) {\(a_2\)};

\node[vtx] (c2F) at (1.45,1.50) {};
\node[anchor=south west] at (1.55,1.35) {\(c_2\)};
\node[vtx,label=below:\(c_1\)] (c1F) at (1.45,0.30) {};

\node[vtx,label=right:\(b\)] (bF) at (4.20,0.90) {};

\draw[newedge] (c1F)--(c2F);
\draw[newedge] (a1F)--(c2F);
\draw[newedge] (a2F)--(c1F);
\draw[newedge] (c1F)--(bF);
\draw[newedge]
    (a1F) .. controls (-3.10,-1.10) and (4.25,-1.05) .. (bF);

\node[paneltitle] at (1.45,-1.65) {{\bfseries (f)} \(G_{u,v}^{1}\)};

\end{scope}

\end{tikzpicture}%
}
\caption{Illustration of the auxiliary graphs \(G_{u,v}^{0}\) and \(G_{u,v}^{1}\).}
\label{fig:example-G0-G1}
\end{figure}

\begin{lemma}[\cite{Moffatt2023}]
	\label{lem:edge-pivot-ppt}
	Let \(G\) be a simple graph and \(uv\in E(G)\). Then
	\[
	\mathbf{A}_G*\{u,v\}=\mathbf{A}_{G\wedge uv}.
	\]
\end{lemma}

\begin{lemma}
\label{lem:pivot-deletion-matrix}
Let \(G\) be a simple graph and \(uv\in E(G)\). Set \(W=V(G)\setminus\{u,v\}\).
Let \(p,q\in \GF(2)^W\) be column vectors satisfying: for all \(x\in W\), \(p_x=1\) if and only if \(x\) is adjacent to \(u\) in \(G\), and \(q_x=1\) if and only if \(x\) is adjacent to \(v\) in \(G\).
Then, for every \(Q\subseteq W\),

\[
\mathbf{A}\bigl(G_{u,v}^{0},Q\bigr)
=
\mathbf{A}(G[W],Q)+pq^{\mathsf T}+qp^{\mathsf T}.
\]
\end{lemma}

\begin{proof}
Order the vertices so that \(u\) and \(v\) come first, followed by the vertices of \(W\) in any fixed order. Then
\[
\mathbf A_G = \begin{pmatrix}
0 & 1 & p^{\mathsf T}\\
1 & 0 & q^{\mathsf T}\\
p & q & \mathbf A_{G[W]}
\end{pmatrix}.
\]
 The leading $2\times2$ block $B = \begin{pmatrix}0&1\\1&0\end{pmatrix}$ is nonsingular over $\GF(2)$ and satisfies $B^{-1}=B$. 
By Lemma~\ref{lem:edge-pivot-ppt}, $\mathbf A_{G\wedge uv} = \mathbf A_G * \{u,v\}$. Then the submatrix of \(\mathbf A_{G\wedge uv}\) indexed by \(W\) is the Schur complement of \(B\) in \(\mathbf A_G\). Hence
\[
\begin{aligned}
\mathbf A_{(G\wedge uv)[W]}
&= \mathbf A_{G[W]} - \begin{pmatrix}p & q\end{pmatrix} B^{-1} \begin{pmatrix}p^{\mathsf T}\\ q^{\mathsf T}\end{pmatrix} \\
&= \mathbf A_{G[W]} - (p q^{\mathsf T} + q p^{\mathsf T}) \\
&= \mathbf A_{G[W]} + p q^{\mathsf T} + q p^{\mathsf T}.
\end{aligned}
\]
Since $(G\wedge uv)[W] = (G\wedge uv)-\{u,v\}=G_{u,v}^{0}$, we have
\[
\mathbf A_{G_{u,v}^{0}} = \mathbf A_{G[W]} + p q^{\mathsf T} + q p^{\mathsf T}.
\]
Thus, for $Q\subseteq W$, \[
\mathbf A\bigl(G_{u,v}^{0}, Q\bigr)
= \mathbf A(G[W],Q) + p q^{\mathsf T}+q p^{\mathsf T}.
\]
\end{proof}

\begin{theorem}
	\label{thm:recurrence-edge}
	Let \(G\) be a simple graph and \(uv\in E(G)\). Then
	\[
	P_G^\times(z)
	=
	z P^\times_{G\underline{*}u}(z)
	+
	z^2\left(
	P_{G_{u,v}^{0}}^\times(z)
	+
	P_{G_{u,v}^{1}}^\times(z)
	\right).
	\]
\end{theorem}

\begin{proof}
Since
\[
P_G^\times(z)
=
\sum_{F\subseteq V(G)}
z^{\operatorname{rank}\mathbf{A}(G,F)},
\]
we split the sum according to whether \(u\in F\). Write
\[
P_G^\times(z)=\Sigma_1+\Sigma_0,
\]
where
\[
\Sigma_1=
\sum_{\substack{F\subseteq V(G)\\ u\in F}}
z^{\operatorname{rank}\mathbf{A}(G,F)},
\qquad
\Sigma_0=
\sum_{\substack{F\subseteq V(G)\\ u\notin F}}
z^{\operatorname{rank}\mathbf{A}(G,F)}.
\]

\noindent{\bf Computation of \(\Sigma_1\).} 
Let \(F\subseteq V(G)\) with \(u\in F\). By Lemma~\ref{lem:local-complement-rank} and the definition of local complementation deletion,
\[
\operatorname{rank}\mathbf{A}(G,F)=\operatorname{rank}\mathbf{A}\bigl(G\underline{*}u,\,(F\setminus\{u\})\Delta N_G(u)\bigr)+1.
\]
Hence, setting \(R=F\setminus\{u\}\),
\[
\Sigma_1 = z \sum_{R\subseteq V(G)\setminus\{u\}} z^{\operatorname{rank}\mathbf{A}(G\underline{*}u,\,R\Delta N_G(u))}
= z P_{(G\underline{*}u,N_G(u))}^{\times}(z)
= z P_{G\underline{*}u}^{\times}(z).
\]

\noindent{\bf Computation of \(\Sigma_0\).}  Suppose \(u\notin F\). Set
\(W=V(G)\setminus\{u,v\}\). Let \(p,q\in \GF(2)^W\) be column vectors satisfying: for all \(x\in W\), \(p_x=1\) if and only if \(x\) is adjacent to \(u\) in \(G\), and \(q_x=1\) if and only if \(x\) is adjacent to \(v\) in \(G\). Set
\(R=F\setminus\{v\}\),
and define
\[
\epsilon=
\begin{cases}
0, & v\notin F,\\
1, & v\in F.
\end{cases}
\]
Then
\[
\mathbf{A}(G,F)=
\begin{pmatrix}
0 & 1 & p^{\mathsf T}\\
1 & \epsilon & q^{\mathsf T}\\
p & q & \mathbf{A}(G[W],R)
\end{pmatrix}
=
\begin{pmatrix}
B_\epsilon & C\\
D & \mathbf{A}(G[W],R)
\end{pmatrix},
\]
where
\[
B_\epsilon=\begin{pmatrix}0&1\\1&\epsilon\end{pmatrix},\quad
C=\begin{pmatrix}p^{\mathsf T}\\ q^{\mathsf T}\end{pmatrix},\quad
D=\begin{pmatrix}p&q\end{pmatrix}.
\]

Then \(B_\epsilon\) is invertible over \(\GF(2)\), and
\(B_\epsilon^{-1}=\begin{pmatrix}\epsilon&1\\1&0\end{pmatrix}\).
By the Schur complement rank formula,
\[
\operatorname{rank}\mathbf{A}(G,F)=2+\operatorname{rank}M_\epsilon,
\] where \[M_\epsilon=\mathbf{A}(G[W],R)+DB_\epsilon^{-1}C.\]
Thus
\[
M_\epsilon=\mathbf{A}(G[W],R)+\epsilon pp^{\mathsf T}+pq^{\mathsf T}+qp^{\mathsf T}.
\]

We now distinguish the two possible values of \(\epsilon\). If \(v\notin F\), then \(\epsilon=0\). Hence
\[
M_0=\mathbf{A}(G[W],R)+pq^{\mathsf T}+qp^{\mathsf T}.
\]
By Lemma~\ref{lem:pivot-deletion-matrix}, we have
\(M_0
=
\mathbf{A}(G_{u,v}^{0},R).
\)
Since
\(
\operatorname{rank}\mathbf{A}(G,F)
=
2+\operatorname{rank}M_0,
\)
the contribution of all subsets \(F\) with \(u\notin F\) and
\(v\notin F\) is
\(
z^2P_{G_{u,v}^{0}}^\times(z).
\)

It remains to consider the case \(v\in F\). Then \(\epsilon=1\), and
\[
M_1=\mathbf{A}(G[W],R)+pp^{\mathsf T}+pq^{\mathsf T}+qp^{\mathsf T}.
\]
Let
\(N=N_G(u)\setminus\{v\}\).
Since \(p_x=1\) if and only if \(x\in N\), the off-diagonal part of
\(pp^{\mathsf T}\) toggles precisely the adjacencies inside \(N\). This is exactly
the effect of the local complementation at \(u\), restricted to \(W\). Moreover, the diagonal part of \(pp^{\mathsf T}\) changes
exactly the diagonal entries corresponding to vertices in \(N\). Hence
\[
\mathbf{A}(G[W],R)+pp^{\mathsf T}
=
\mathbf{A}((G*u)[W],R\Delta N).
\]
By Lemma~\ref{lem:pivot-deletion-matrix}, we have 
\[
\begin{aligned}
M_1&=\mathbf{A}(G[W],R)+pp^{\mathsf T}+pq^{\mathsf T}+qp^{\mathsf T}\\
&=\mathbf{A}((G*u)[W],R\Delta N)+pq^{\mathsf T}+qp^{\mathsf T}\\
&=
\mathbf{A}(G_{u,v}^{1},R\Delta N).
\end{aligned}
\]
Since \(\operatorname{rank}\mathbf{A}(G,F)
=2+\operatorname{rank}M_1,\)
the contribution of all subsets \(F\) with \(u\notin F\) and \(v\in F\)
is \( z^2P_{G_{u,v}^{1}}^\times(z).\)

Combining the three contributions, we get
\[
P_G^\times(z)
=
z P^\times_{G\underline{*}u}(z)
+
z^2\left(
P_{G_{u,v}^{0}}^\times(z)
+
P_{G_{u,v}^{1}}^\times(z)
\right).
\]
\end{proof}

\begin{remark}\label{rem:leaf}
Let \(u\) be a leaf of \(G\) with \(N_G(u)=\{v\}\). Then \(N_G(u)\setminus\{v\}=\emptyset\), and by definition we have
\[
G_{u,v}^{0}=G_{u,v}^{1}=G\setminus\{u,v\},\qquad G\underline{*}u = G\setminus \{u\}.
\]
Consequently, Theorem~\ref{thm:recurrence-edge} reduces to
\[
\ptr_G(z)=z\,\ptr_{G\setminus \{u\}}(z)+2z^2\ptr_{G\setminus\{u,v\}}(z),
\]
which is exactly the leaf-reduction formula obtained by Deng, Jin and Yan \cite{DengJinYan2026}. Thus Theorem~\ref{thm:recurrence-edge} extends this leaf formula to vertices of arbitrary degree.
\end{remark}

\begin{example}\label{ex:complete-graphs}
Let \(E_m\) denote the edgeless graph on \(m\) vertices. Let \(K_n\) be the complete graph on \(n\geq 2\) vertices and let \(uv\) be any edge. Then
\[
(K_n)\underline{*}u = E_{n-1},\qquad
(K_n)_{u,v}^{0}=K_{n-2},\qquad
(K_n)_{u,v}^{1}=E_{n-2}.
\]
Moreover, \(\ptr_{E_{m}}(z)=(1+z)^m\). By Theorem~\ref{thm:recurrence-edge}, we have
\[
\ptr_{K_n}(z)=z(1+z)^{n-1}+z^2\bigl(\ptr_{K_{n-2}}(z)+(1+z)^{n-2}\bigr),
\]
with initial values \(\ptr_{K_1}(z)=1+z\) and \(\ptr_{K_2}(z)=z+3z^2\). Thus
\[
\ptr_{K_n}(z)
=
\begin{cases}
\displaystyle
z^n+\sum_{i=1}^{n}\binom{n}{\,n+1-i\,}z^i,
& n\text{ even},\\[1.2ex]
\displaystyle
z^{n-1}+\sum_{i=1}^{n}\binom{n}{\,n+1-i\,}z^i,
& n\text{ odd}.
\end{cases}
\]
This agrees with the formula for complete graphs obtained by Yan and Li
\cite[Theorem~16]{YanLi2025}. 
\end{example}

\section{Comparing the lowest and highest coefficients}

For a simple graph \(G\), let \(\rho(G)\) be the lowest degree of \(\ptr_G(z)\). Denote by \(\ell_G\) the coefficient of this lowest-degree term and by \(h_G\) the coefficient of the highest-degree term.

\begin{lemma}[{\cite{DengJinYan2026}}]\label{lem:highest-degree}
Let $G$ be a simple graph with $n$ vertices. Then the highest degree of $\ptr_{G}(z)$ is $n$.
\end{lemma}

\begin{theorem}\label{thm:extremal-coefficients}
	Let \(G\) be a simple graph. Then
	\[
	\ell_{G}\leq h_{G}.
	\]
	Moreover, $\ell_{G}= h_{G}$ if and only if \(G\) is edgeless.
\end{theorem}

\begin{proof}
	If \(G\) is edgeless, then \(\ptr_G(z)=(1+z)^{|V(G)|}\), so \(\ell_G=h_G=1\). It remains to prove that \(\ell_G<h_G\) whenever \(G\) has at least one edge. 
    
    We prove this by induction on \(n=|V(G)|\). The cases \(n=0\) and \(n=1\) are vacuous. Let \(n\geq 2\), and assume that the strict inequality holds for every simple graph with fewer than \(n\) vertices and at least one edge. Together with the edgeless case, this implies that for any graph with fewer than \(n\) vertices, we have \(\ell\leq h\), with strict inequality whenever the graph has at least one edge.

	Now let \(G\) be a simple graph on \(n\) vertices with at least one edge.
Choose $uv\in E(G)$.
By Theorem~\ref{thm:recurrence-edge},
\[
P_G^\times(z)
=
z P^\times_{G\underline{*}u}(z)
+
z^2\left(
P_{G_{u,v}^{0}}^\times(z)
+
P_{G_{u,v}^{1}}^\times(z)
\right).
\]
The graph \(G\underline{*}u\) has \(n-1\) vertices, while
\(G_{u,v}^{0}\) and \(G_{u,v}^{1}\) have \(n-2\) vertices. By Lemma~\ref{lem:highest-degree}, after multiplying by \(z\) and \(z^2\) respectively, the three constituent terms on the right-hand side all have highest degree \(n\).
	Since all coefficients are nonnegative, no cancellation occurs in the
	highest-degree term. Hence
	\[
	h_G
	=
	h_{G\underline{*}u}
	+
	h_{G_{u,v}^{0}}
	+
	h_{G_{u,v}^{1}}.
	\]
    
	Let
	\(d_1=\rho(G\underline{*}u)+1,
	d_2=\rho(G_{u,v}^{0})+2\), and \(
	d_3=\rho(G_{u,v}^{1})+2.
	\)
	These are precisely the lowest degrees of the three constituent terms in the
	recurrence formula.
    
	 If \(d_1,d_2\) and \(d_3\) are not all equal, then at least one constituent term does
	not contribute to the lowest-degree term of \(P_G^\times(z)\). Since each
	lowest-degree coefficient is positive, we have
	\[
	\ell_G
	<
	\ell_{G\underline{*}u}
	+
	\ell_{G_{u,v}^{0}}
	+
	\ell_{G_{u,v}^{1}}.
	\]
	By the induction hypothesis together with the edgeless case,
	\[
	\ell_{G\underline{*}u}\leq h_{G\underline{*}u},\qquad
	\ell_{G_{u,v}^{0}}\leq h_{G_{u,v}^{0}},\qquad
	\ell_{G_{u,v}^{1}}\leq h_{G_{u,v}^{1}}.
	\]
	Therefore
	\[
	\ell_G
	<
	h_{G\underline{*}u}
	+
	h_{G_{u,v}^{0}}
	+
	h_{G_{u,v}^{1}}
	=
	h_G.
	\]
	
	It remains to consider the case \(
	d_1=d_2=d_3.\)
	Then
	\[
	\rho(G\underline{*}u)+1
	=
	\rho(G_{u,v}^{0})+2
	=
	\rho(G_{u,v}^{1})+2,
	\]
	and hence
	\[
	\rho(G\underline{*}u)
	=
	\rho(G_{u,v}^{0})+1
	=
	\rho(G_{u,v}^{1})+1.
	\]
	Since $\rho(G_{u,v}^{0})\geq0$, we have
	\(
	\rho(G\underline{*}u)\geq 1.
	\)
	Suppose \(G\underline{*}u\) is edgeless. Then 
	\[
	P^\times_{G\underline{*}u}(z)
	=
	(1+z)^{|V(G\underline{*}u)|},
	\]
	and hence \(\rho(G\underline{*}u)=0\), a contradiction. Thus \(G\underline{*}u\) contains at least one edge. Since
	\(G\underline{*}u\) has fewer than \(n\) vertices, the induction hypothesis
	gives
	\(
	\ell_{G\underline{*}u}<h_{G\underline{*}u}.
	\)
	In the present case, all three constituent terms contribute to the
	lowest-degree term of \(P_G^\times(z)\). Hence
	\[
	\ell_G
	=
	\ell_{G\underline{*}u}
	+
	\ell_{G_{u,v}^{0}}
	+
	\ell_{G_{u,v}^{1}}.
	\]
	Together with
	\[
	\ell_{G\underline{*}u}<h_{G\underline{*}u},\qquad
	\ell_{G_{u,v}^{0}}\leq h_{G_{u,v}^{0}},\qquad
	\ell_{G_{u,v}^{1}}\leq h_{G_{u,v}^{1}},
	\]
	we obtain
	\[
	\ell_G
	<
	h_{G\underline{*}u}
	+
	h_{G_{u,v}^{0}}
	+
	h_{G_{u,v}^{1}}
	=
	h_G.
	\]
    
	Thus \(\ell_G<h_G\) whenever \(G\) has at least one edge.
\end{proof}

\begin{remark}
Theorem~\ref{thm:extremal-coefficients} implies that \(\ptr_G(z)\) is symmetric (i.e., its coefficients from the lowest degree to the highest degree form a palindromic sequence) if and only if \(G\) is edgeless.
\end{remark}

\section*{Acknowledgements}
This work is supported by NSFC (Nos. 12471326, 12401474).

\end{document}